\documentclass{bcp-void}
\usepackage{hyperref,color}
\usepackage{amsmath,amsthm}
\usepackage{amssymb}

\newtheorem{theorem}{Theorem}[section]

\newtheorem{lemma}[theorem]{Lemma}

\theoremstyle{definition}

\newtheorem{question}[theorem]{Question}
\numberwithin{equation}{section}

\newcommand{\C}{\mathbb{C}}
\newcommand{\N}{\mathbb{N}}
\newcommand{\R}{\mathbb{R}}
\newcommand{\Z}{\mathbb{Z}}

\newcommand{\cA}{\mathcal{A}}
\newcommand{\cB}{\mathcal{B}}
\newcommand{\cD}{\mathcal{D}}
\newcommand{\cE}{\mathcal{E}}
\newcommand{\cK}{\mathcal{K}}
\newcommand{\cM}{\mathcal{M}}

\newcommand{\fM}{\mathfrak{M}}

\newcommand{\eps}{\varepsilon}
\newcommand{\dR}{{\bf\dot{\R}}}

\begin{document}


\mathclass{Primary 47G10; Secondary 46E30, 42C40.}
\abbrevauthors{C. A. Fernandes, A. Yu. Karlovich, Yu. I. Karlovich}
\abbrevtitle{Algebra of convolution type operators}

\title{Algebra of convolution type operators
with continuous data on Banach function spaces}

\author{Cl\'audio A. Fernandes}
\address{%%
Centro de Matem\'atica e Aplica\c{c}\~oes,
Departamento de Matem\'atica,\\
Faculdade de Ci\^encias e Tecnologia,
Universidade Nova de Lisboa,\\
Quinta da Torre,
2829--516 Caparica,
Portugal\\
E-mail: caf@fct.unl.pt}

\author{Alexei Yu. Karlovich}
\address{%%
Centro de Matem\'atica e Aplica\c{c}\~oes,
Departamento de Matem\'atica,\\
Faculdade de Ci\^encias e Tecnologia,
Universidade Nova de Lisboa,\\
Quinta da Torre,
2829--516 Caparica,
Portugal\\
E-mail: oyk@fct.unl.pt}

\author{Yuri I. Karlovich}
\address{%
Centro de Investigaci\'on en Ciencias,\\
Instituto de Investigaci\'on en Ciencias B\'asicas y Aplicadas,\\
Universidad Aut\'onoma del Estado de Morelos,\\
Av. Universidad 1001, Col. Chamilpa,\\
C.P. 62209 Cuernavaca, Morelos, M\'exico\\
E-mail: karlovich@uaem.mx}

\maketitlebcp

%%%----------------------------------------------------------------------------
\begin{abstract}
We show that if the Hardy-Littlewood maximal operator is bounded on a 
reflexive Banach function space $X(\R)$ and on its associate space $X'(\R)$,
then the space $X(\R)$ has an unconditional wavelet basis. As a consequence 
of the existence of a Schauder basis in $X(\R)$, we prove that the ideal of 
compact operators $\cK(X(\R))$ on the space $X(\R)$ is contained in the 
Banach algebra generated by all operators of multiplication $aI$ by functions 
$a\in C(\dR)$, where $\dR=\R\cup\{\infty\}$, and by all Fourier convolution 
operators $W^0(b)$ with symbols $b\in C_X(\dR)$, the Fourier multiplier 
analogue of $C(\dR)$.
\end{abstract}
%%%----------------------------------------------------------------------------
\section{Introduction}
The set of all Lebesgue measurable complex-valued functions on $\R$ is denoted
by $\fM(\R)$. Let $\fM^+(\R)$ be the subset of functions in $\fM(\R)$ whose
values lie  in $[0,\infty]$. For a measurable set $E\subset\R$, 
its Lebesgue measure and the characteristic function are denoted by $|E|$ and
$\chi_E$, respectively. Following \cite[Chap.~1, Definition~1.1]{BS88}, a 
mapping $\rho:\fM^+(\R)\to [0,\infty]$ is called a Banach function norm if,
for all functions $f,g, f_n \ (n\in\N)$ in $\fM^+(\R)$, for all
constants $a\ge 0$, and for all measurable subsets $E$ of $\R$,
the following properties hold:
\begin{eqnarray*}
{\rm (A1)} &\quad & \rho(f)=0  \Leftrightarrow  f=0\ \mbox{a.e.}, \
\rho(af)=a\rho(f), \
\rho(f+g) \le \rho(f)+\rho(g),\\
{\rm (A2)} &\quad &0\le g \le f \ \mbox{a.e.} \ \Rightarrow \ \rho(g)
\le \rho(f)
\quad\mbox{(the lattice property)},
\\
{\rm (A3)} &\quad &0\le f_n \uparrow f \ \mbox{a.e.} \ \Rightarrow \
       \rho(f_n) \uparrow \rho(f)\quad\mbox{(the Fatou property)},\\
{\rm (A4)} &\quad & |E|<\infty \Rightarrow \rho(\chi_E) <\infty,\\
{\rm (A5)} &\quad & |E|<\infty \Rightarrow \int_E f(x)\,dx \le C_E\rho(f),
\end{eqnarray*}
%%%%
where $C_E \in (0,\infty)$ may depend on $E$ and $\rho$ but is
independent of $f$. When functions differing only on a set of measure zero
are identified, the set $X(\R)$ of functions $f\in\fM(\R)$
for which $\rho(|f|)<\infty$ is called a Banach function space. For each
$f\in X(\R)$, the norm of $f$ is defined by
\[
\left\|f\right\|_{X(\R)} :=\rho(|f|).
\]
With this norm and under natural linear space operations, the set $X(\R)$ 
becomes a Banach space (see \cite[Chap.~1, Theorems~1.4 and~1.6]{BS88}). 
If $\rho$ is a Banach function norm, its associate norm $\rho'$ is defined on
$\fM^+(\R)$ by
\[
\rho'(g):=\sup\left\{
\int_{\R} f(x)g(x)\,dx \ : \ f\in \fM^+(\R), \ \rho(f) \le 1
\right\}, \quad g\in \fM^+(\R).
\]
By \cite[Chap.~1, Theorem~2.2]{BS88}, $\rho'$ is itself 
a Banach function norm.
The Banach function space $X'(\R)$ determined by the Banach function norm
$\rho'$ is called the associate space (K\"othe dual) of $X(\R)$.
The associate space $X'(\R)$ is a subspace of the (Banach) dual
space $[X(\R)]^*$.

Let $F:L^2(\R)\to L^2(\R)$ denote the Fourier transform
\[
(Ff)(x):=\widehat{f}(x):=\int_\R f(t)e^{itx}\,dt,
\quad
x\in\R,
\]
and let $F^{-1}:L^2(\R)\to L^2(\R)$ be the inverse of $F$,
\[
(F^{-1}g)(t)=\frac{1}{2\pi}\int_\R g(x)e^{-itx}\,d x,
\quad
t\in\R.
\]
It is well known that the Fourier convolution operator $W^0(a):=F^{-1}aF$
is bounded on the space $L^2(\R)$ for every $a\in L^\infty(\R)$.
Let $X(\R)$ be a separable Banach function space. Then $L^2(\R)\cap X(\R)$
is dense in $X(\R)$ (see Lemma~\ref{le:density} below). A function 
$a\in L^\infty(\R)$ is called a Fourier multiplier on $X(\R)$ if the 
convolution operator $W^0(a):=F^{-1}aF$ maps $L^2(\R)\cap X(\R)$ into 
$X(\R)$ and extends to a bounded linear operator on $X(\R)$. The function 
$a$ is called the symbol of the Fourier convolution operator $W^0(a)$. 
The set $\cM_{X(\R)}$ of all Fourier multipliers on  $X(\R)$ is a unital 
normed algebra under pointwise operations and the norm
\[
\left\|a\right\|_{\cM_{X(\R)}}:=\left\|W^0(a)\right\|_{\cB(X(\R))},
\]
where $\cB(X(\R))$ denotes the Banach algebra of all bounded linear operators
on the space $X(\R)$. Let $\cK(X(\R))$ denote the ideal of all compact
operators in the Banach algebra $\cB(X(\R))$.

Recall that the (non-centered) Hardy-Littlewood maximal function $Mf$ of a
function $f\in L_{\rm loc}^1(\R)$ is defined by
\[
(Mf)(x):=\sup_{Q\ni x}\frac{1}{|Q|}\int_Q|f(y)|\,dy,
\]
where the supremum is taken over all intervals $Q\subset\R$ of finite length
containing $x$. The Hardy-Littlewood maximal operator $M$ defined by the rule 
$f\mapsto Mf$ is a sublinear operator.

Suppose $X(\R)$ is a separable Banach function space such that the 
Hardy-Littlewood maximal operator $M$ is bounded on the space $X(\R)$
and on its associate space $X'(\R)$. Let $C(\dR)$ denote the $C^*$-algebra of 
continuous functions on the one-point compactification $\dR=\R\cup\{\infty\}$ 
of the real line. Further, let $C_X(\dR)$ be the closure of 
$C(\dR)\cap V(\R)$ in the norm of $\cM_{X(\R)}$, where $V(\R)$ is the algebra 
of all functions of finite total variation on $\R$. Consider the smallest 
Banach subalgebra 
\[
\cA_{X(\R)}=\operatorname{alg}\{aI,W^0(b)\ :\ a\in C(\dR),\ b\in C_X(\dR)\}
\]
of the algebra $\cB(X(\R))$ that contains all operators of multiplication $aI$ 
by functions $a\in C(\dR)$ and all Fourier convolution operators $W^0(b)$ with 
symbols $b\in C_X(\dR)$. 

The algebra $\cA_{X(\R)}$ is well understood in the case when $X(\R)=L^p(\R,w)$ 
is a Lebesgue space with $1<p<\infty$ and a Muckenhoupt weight $w$ (see, e.g., 
\cite[Chap.~17]{BKS02} and also \cite{D79} for the non-weighted case). 
Surprisingly enough, the algebra $\cA_{X(\R)}$ has not been investigated for 
more general Banach function spaces $X(\R)$. The aim of this paper is to start 
studying the algebra $\cA_{X(\R)}$ on reflexive Banach function spaces $X(\R)$ 
under the assumption that the Hardy-Littlewood maximal operator $M$ is 
bounded on the space $X(\R)$ and on its associate space $X'(\R)$. 

Our main result is the following.
%%%----------------------------------------------------------------------------
\begin{theorem}\label{th:main}
Let $X(\R)$ be a reflexive Banach function space such that the Hardy-Littlewood 
maximal operator $M$ is bounded on $X(\R)$ and on its associate space $X'(\R)$. 
Then the ideal of compact operators $\cK(X(\R))$ is contained in the Banach
algebra $\cA_{X(\R)}$.
\end{theorem}
%%%----------------------------------------------------------------------------
Theorem~\ref{th:main} implies that the quotient Banach algebra 
\[
\cA_{X(\R)}^\pi:=\cA_{X(\R)}/\cK(X(\R))
\]
is well-defined. It follows from \cite[Theorem~2.9]{K15b} that if
$X(\R)$ is either a reflexive rearrangement-invariant Banach function
space with nontrivial Boyd indices or a reflexive variable Lebesgue space
such that the Hardy-Littlewood maximal operator is bounded on $X(\R)$,
then $\cA_{X(\R)}^\pi$ is commutative. 
%%%----------------------------------------------------------------------------
\begin{question}
Is it true that the quotient algebra $\cA_{X(\R)}^\pi$ is commutative
under the assumptions of Theorem~\ref{th:main}?
\end{question}
%%%----------------------------------------------------------------------------
In order to prove Theorem~\ref{th:main}, we have to insure that the space 
$X(\R)$ has a Schauder basis. In Section~\ref{sec:wavelet}, we prove a 
stronger result that might be of independent interest. It says that, under 
the assumptions of Theorem~\ref{th:main}, the space $X(\R)$ has an 
unconditional wavelet basis. Similar questions were considered earlier in 
\cite{NPR14} and \cite{INS15} under hypotheses on the space $X(\R)$, which 
are different from ours (see also \cite{Ho11a,Ho11b,So97,W12}).

In Section~\ref{sec:noncompactness}, we observe that the multiplication 
operators $aI$ with $a\in L^\infty$ and the Fourier convolution operators
$W^0(b)$ with $b\in\cM_{X(\R)}$ cannot be compact on the space $X(\R)$ unless 
they are trivial. Thus, the nontirival generators of the algebra $\cA_{X(\R)}$ 
are noncompact.  

In Section~\ref{sec:algebra}, we state that a rank one
operator $T_1$ defined by 
\[
(T_1f)(x)=a(x)\int_{\R} b(y)f(y)\,dy, 
\]
where $a$ and $b$ are continuous and compactly supported functions, can be 
written as a product of generators of the algebra $\cA_{X(\R)}$. We prove 
Theorem~\ref{th:main} by showing that each compact operator can be 
approximated in the operator norm by finite rank operators and, further, 
by a finite sum of operators of the form $T_1$.
%%%----------------------------------------------------------------------------
\section{Wavelet bases in Banach function spaces}\label{sec:wavelet}
%%%----------------------------------------------------------------------------
\subsection{Density of nice functions in separable Banach function spaces}
Let $C_0(\R)$ and $C_0^\infty(\R)$ denote the sets of continuous compactly
supported functions on $\R$ and infinitely differentiable compactly supported
functions on $\R$, respectively.
%%%----------------------------------------------------------------------------
\begin{lemma}\label{le:density}
Let $X(\R)$ be a separable Banach function space. Then the sets $C_0(\R)$,
$C_0^\infty(\R)$ and $L^2(\R)\cap X(\R)$ are dense in the space $X(\R)$.
\end{lemma}
%%%----------------------------------------------------------------------------
The density of $C_0(\R)$ and $C_0^\infty(\R)$ in $X(\R)$ is shown in 
\cite[Lemma~2.12]{KS14}. Since $C_0(\R)\subset L^2(\R)\cap X(\R)\subset X(\R)$,
we conclude that $L^2(\R)\cap X(\R)$ is dense in $X(\R)$.
%%%----------------------------------------------------------------------------
\subsection{Uniform boundedness of families of operators satisfying
local sharp maximal operator estimates uniformly}
For $s>0$ and $f\in L^s_{\rm loc}(\R)$, consider the local $s$-sharp maximal 
function of $f$ defined by
\[
f_s^\#(x):=\sup_{Q\ni x}\inf_{c\in\C}
\left(\frac{1}{|Q|}\int_Q|f(y)-c|^s\,dy\right)^{1/s},
\]
where the supremum is taken over all intervals $Q\subset\R$ of finite length
containing $x$.

The theorem below follows from \cite[Theorem~3.6]{KS14}.
%%%----------------------------------------------------------------------------
\begin{theorem}\label{th:uniform-boundedness}
Let $X(\R)$ be a separable Banach function space such that the Hardy-Littlewood
maximal operator $M$ is bounded on $X(\R)$ and on its associate space $X'(\R)$. 
Assume that $0<s<1$ and $\Omega$ is an index set. Let 
$\{T_\omega\}_{\omega\in\Omega}$ be a family of linear operators such that
%%%
\begin{enumerate}
\item[{\rm(a)}]
for each $\omega\in\Omega$, the operator $T_\omega$ is bounded on the space 
$L^2(\R)$;

\item[{\rm (b)}]
there exists a constant $c_s\in(0,\infty)$ depending only on $s$ and such that 
for every $\omega\in\Omega$, every $f\in C_0^\infty(\R)$ and every $x_0\in\R$, 
one has
\[
(T_\omega f)_s^\#(x_0)\le c_s(Mf)(x_0).
\]
\end{enumerate}
%%%
Then each operator $T_\omega$, $\omega\in\Omega$, is bounded on $X(\R)$ and
\[
\sup_{\omega\in\Omega}\|T_\omega\|_{\cB(X(\R))}<\infty.
\]
\end{theorem}
%%%----------------------------------------------------------------------------
\subsection{Estimates for local sharp maximal operators of families of
operators associated with kernels}
Let $\cD'(\R)$ be the space of distributions, that is, the dual space of 
$C_0^\infty(\R)$. The action of a distribution $a\in\cD'(\R)$ on a function
$f\in C_0^\infty(\R)$ is denoted by $a(f)=\langle a,f\rangle$. A locally 
integrable function $K:\R^2\setminus\{(x,x):x\in\R\}\to\C$ is said to be a 
kernel. One says that a linear and continuous operator
$T_K:C_0^\infty(\R)\to\cD'(\R)$ is associated with a kernel $K$ if
\[
\langle T_Kf,g\rangle=\int_\R\int_\R K(x,y)g(x)f(y)\,dx\,dy
\]
whenever $f,g\in C_0^\infty(\R)$ with 
$\operatorname{supp}f\cap\operatorname{supp}g=\emptyset$.

For each point $x_0\in\R$, each radius $r>0$ and a kernel $K$, we consider 
the interval
\[
I=I(x_0,r)=(x_0-r,x_0+r)
\]
and the function
\[
(D_IK)(y):=\frac{1}{|I|^2}\iint_{I\times I}|K(z,y)-K(x,y)|\,dx\,dz.
\]

Let $\Omega$ be an index set. Following \cite[Section~2.1]{AP94}, a family 
of kernels $\{K_\omega\}_{\omega\in\Omega}$ is said to satisfy Condition $(D)$
uniformly in $\Omega$ if there are constants $C_D,N\in(0,\infty)$ such that
for all $\omega\in\Omega$, all $f\in C_0^\infty(\R)$ and all $x_0\in\R$,
%%%
\begin{equation}\label{eq:condition-CD}
\sup_{r>0}\int_{|y-x_0|>Nr}(D_IK_\omega)(y)|f(y)|\,dy\le C_D(Mf)(x_0).
\end{equation}
%%%----------------------------------------------------------------------------
\begin{theorem}\label{th:Alvarez-Perez-uniform}
Let $\Omega$ be an index set, $\{K_\omega\}_{\omega\in\Omega}$ be a family
of kernels and let $\{T_{K_\omega}\}_{\omega\in\Omega}$ be the family of 
operators associated with the kernels in the family 
$\{K_\omega\}_{\omega\in\Omega}$. If
%%%
\begin{enumerate}
\item[{\rm(a)}]
the family $\{K_\omega\}_{\omega\in\Omega}$ satisfies Condition $(D)$ uniformly
in $\Omega$ with some constants $C_D,N\in(0,\infty)$;

\item[{\rm(b)}]
the operators $T_{K_\omega}$ extend to bounded operators from $L^1(\R)$ into 
$L^{1,\infty}(\R)$ uniformly in $\Omega$, that is, there exists a constant 
$C_{1,1}\in(0,\infty)$ such that for all $\omega\in\Omega$ and all 
$f\in C_0^\infty(\R)$,
\[
\sup_{\lambda>0}\left(\lambda\left|\left\{x\in\R\ :\ 
\left|\left(T_{K_\omega}f\right)(x)\right|>\lambda\right\}\right|\right)
\le 
C_{1,1}\|f\|_{L^1(\R)},
\]
\end{enumerate}
then for all $\omega\in\Omega$, all $s\in(0,1)$, all $f\in C_0^\infty(\R)$ 
and all $x_0\in\R$,
%%%
\begin{equation}\label{eq:Alvarez-Perez-uniform}
\left(T_{K_\omega}f\right)_s^\#(x_0)\le C_s(Mf)(x_0),
\end{equation}
%%%
where
%%%
\[
C_s:=2^{2/s-1}(N(1-s)^{-1/s}C_{1,1}+C_D).
\]
%%%
\end{theorem}
%%%----------------------------------------------------------------------------
\begin{proof}
This theorem is proved by analogy with \cite[Theorem~2.1]{AP94} (see also
\cite[Theorem~2.6]{KJLH09} for tracing the constant $C_s$). Since the 
definition of the local $s$-sharp maximal function adopted in this paper
slightly differs from that of \cite{AP94, KJLH09}, we provide some details 
here.

Fix $\omega\in\Omega$, $f\in C_0^\infty(\R)$ and $x_0\in\R$. Let $Q$ be an 
interval of finite length containing $x_0$ and $I$ be the smallest interval 
centered at $x_0$, which contains $Q$. Then $|Q|\le|I|\le 2|Q|$ and
%%%
\begin{equation}\label{eq:Alvarez-Perez-uniform-1}
\inf_{c\in\C}
\left(\frac{1}{|Q|}
\int_Q |(T_{K_\omega}f)(y)-c|^s\,dy
\right)^{1/s}
\le 
2^{1/s}\inf_{c\in\C}
\left(\frac{1}{|I|}
\int_I |(T_{K_\omega}f)(y)-c|^s\,dy
\right)^{1/s}.
\end{equation}
%%%
Let $f=f_1+f_2$ where $f_1=f\chi_{I(x_0,Nr)}$, $I=I(x_0,r)$, and $N$ is given 
by Condition $(D)$. Set
\[
(T_{K_\omega}f_2)_I=\frac{1}{|I|}\int_I (T_{K_\omega}f_2)(y)\,dy.
\]
Since $||a|^s-|b|^s|\le |a-b|^s$ and 
$(|a|+|b|)^{1/s}\le 2^{1/s-1}(|a|^{1/s}+|b|^{1/s})$ for $a,b\in\C$ and 
$0<s<1$, we have
%%%
\begin{align}\label{eq:Alvarez-Perez-uniform-2}
\inf_{c\in\C} &\left(\frac{1}{|I|}
\int_I |(T_{K_\omega}f)(y)-c|^s\,dy\right)^{1/s}
\nonumber\\
&\le
\left(\frac{1}{|I|}\int_I 
|(T_{K_\omega}f_1)(y)+(T_{K_\omega}f_2)(y)-(T_{K_\omega}f_2)_I|^s\,dy
\right)^{1/s}
\nonumber\\
&\le 
\left(
\frac{1}{|I|}\int_I|(T_{K_\omega}f_1)(y)|^s\,dy
+
\frac{1}{|I|}\int_I|(T_{K_\omega}f)_2(y)-(T_{K_\omega}f_2)_I|^s\,dy
\right)^{1/s}
\nonumber\\
&\le 
2^{1/s-1}
\left[
\left(\frac{1}{|I|}
\int_I |(T_{K_\omega}f_1)(y)|^s\,dy\right)^{1/s}
+
\left(\frac{1}{|I|}
\int_I|(T_{K_\omega}f_2)(y)-(T_{K_\omega}f_2)_I|^s\,dy\right)^{1/s}
\right]
\nonumber\\
&=: 2^{1/s-1}(J_1+J_2).
\end{align}
%%%
By \cite[formulas (2.48)--(2.49)]{KJLH09}, 
%%%
\begin{equation}\label{eq:Alvarez-Perez-uniform-3}
J_1\le N(1-s)^{-1/s}C_{1,1}(Mf)(x_0),
\quad
J_2\le C_D(Mf)(x_0).
\end{equation}
%%%
Combining \eqref{eq:Alvarez-Perez-uniform-1}--\eqref{eq:Alvarez-Perez-uniform-3},
we arrive at \eqref{eq:Alvarez-Perez-uniform}.
\end{proof}
%%%----------------------------------------------------------------------------
\subsection{Families of standard kernels in the sense of Coifman and Meyer}
Let $\Omega$ be an index set. We say that a family of kernels 
$\{K_\omega\}_{\omega\in\Omega}$ is a uniform in $\Omega$ family of standard 
kernels (in the sense of Coifman and Meyer, see, e.g., \cite[p.~9]{MC97})
if there exist constants $C_1,C_2,C_3\in(0,\infty)$ such that for all
$\omega\in\Omega$ and all pairs $(x,y)$, $(z,y)$, $(x,w)$ in 
$\R^2\setminus\{(x,x):x\in\R\}$, one has
%%%
\begin{align}
& 
|K_\omega(x,y)| \le\frac{C_1}{|x-y|},
\label{eq:CZ-kernel-1}
\\
& 
|K_\omega(z,y)-K_\omega(x,y)|\le\frac{C_2|z-x|}{|x-y|^2}
&\mbox{if}\quad&
|z-x|\le\frac{1}{2}{|x-y|},
\label{eq:CZ-kernel-2}
\\
&
|K_\omega(x,w)-K_\omega(x,y)|\le\frac{C_3|w-y|}{|x-y|^2}
&\mbox{if}\quad&
|w-y|\le\frac{1}{2}{|x-y|}.
\label{eq:CZ-kernel-3}
\end{align}
%%%

An operator $T_{K_\omega}$ associated with a standard kernel $K_\omega$ is 
called a Calder\'on-Zygmund operator.
%%%----------------------------------------------------------------------------
\begin{lemma}\label{le:standard-kernel-implies-CD}
Let $\Omega$ be an index set and let $\{K_\omega\}_{\omega\in\Omega}$ be a 
uniform in $\Omega$ family of standard kernels. Then 
$\{K_\omega\}_{\omega\in\Omega}$ satisfies Condition $(D)$ uniformly in 
$\Omega$ with the constants $C_{D}=8C_2$ and $N=2$.
\end{lemma}
%%%----------------------------------------------------------------------------
\begin{proof}
Take $N=2$. Fix $x_0\in\R$ and $r>0$. If $y\notin I=I(x_0,2r)$, then
for $x,z\in I(x_0,r)$,
\[
|x-x_0|<r\le\frac{1}{2}|y-x_0|,
\quad
|z-x_0|<r\le\frac{1}{2}|y-x_0|.
\]
Hence, taking into account \eqref{eq:CZ-kernel-2}, we obtain for all 
$\omega\in\Omega$ and $x,z\in I(x_0,r)$,
%%%
\begin{align*}
\left|K_\omega(z,y)-K_\omega(x,y)\right|
&\le 
\left|K_\omega(z,y)-K_\omega(x_0,y)\right|
+
\left|K_\omega(x,y)-K_\omega(x_0,y)\right|
\\
&\le 
\frac{C_2|z-x_0|}{|y-x_0|^2}+\frac{C_2|x-x_0|}{|y-x_0|^2}
\le 
\frac{2C_2r}{|y-x_0|^2}.
\end{align*}
%%%
Then
\[
(D_IK_\omega)(y)
=
\frac{1}{|I(x_0,r)|^2}\iint_{I(x_0,r)\times I(x_0,r)}
\left|K_\omega(z,y)-K_\omega(x,y)\right|\,dx\,dz
\le 
\frac{2C_2r}{|y-x_0|^2}
\]
and
%%%
\begin{align*}
\int_{|y-x_0|>2r}(D_IK_\omega)(y)|f(y)|\,dy
&\le 
2C_2r\int_{|y-x_0|>2r}\frac{|f(y)}{|y-x_0|^2}\,dy
\\
&=
2C_2r\sum_{n=0}^\infty\int_{2^{n+1}r<|y-x_0|\le 2^{n+2}r}
\frac{|f(y)|}{|y-x_0|^2}\,dy
\\
&\le 
2C_2r\sum_{n=0}^\infty
\frac{1}{(2^{n+1}r)^2}
\int_{2^{n+1}r<|y-x_0|\le 2^{n+2}r}|f(y)|\,dy
\\
&\le 
2C_2\sum_{n=0}^\infty
\frac{2^{-n}}{2^{n+2}r}
\int_{I(x_0,2^{n+2}r)}|f(y)|\,dy
\\
&=
4C_2\sum_{n=0}^\infty\frac{2^{-n}}{|I(x_0,2^{n+2}r)|}
\int_{I(x_0,2^{n+2}r)}|f(y)|\,dy
\\
&\le 
4C_2\left(\sum_{n=0}^\infty 2^{-n}\right)(Mf)(x_0) 
=8C_2(Mf)(x_0),
\end{align*}
%%%
which implies \eqref{eq:condition-CD} with $N=2$ and $C_D=8C_2$.
\end{proof}
%%%----------------------------------------------------------------------------
\subsection{Families of Calder\'on-Zygmund operators associated with kernels
defined by orthonormal wavelets}
As usual, let
%%%
\begin{equation}\label{eq:pairing}
\langle f,g\rangle:=\int_\R f(x)\overline{g(x)}\,dx.
\end{equation}
%%%
be the standard inner product in $L^2(\R)$.

Following \cite[Section~5.3]{HW96}, a function 
\[
W:[0,\infty)\to(0,\infty)
\]
is said to be a radial decreasing $L^1$-majorant of a function $g:\R\to\C$
if $|g(x)|\le W(|x|)$ for a.e. $x\in\R$, and $W\in L^1([0,\infty))$, $W$ is 
decreasing, $W(0)<\infty$.

Recall that a function $\psi\in L^2(\R)$ is called an orthonormal wavelet
if the family 
\[
\psi_{j,k}(x):=2^{j/2}\psi(2^jx-k),
\quad x\in\R,
\quad j,k\in\Z,
\] 
forms an orthonormal basis in $L^2(\R)$. 

Let $\cE$ be the family of all sequences $\eps=\{\eps_{j,k}\}_{j,k\in\Z}$
with $\eps_{j,k}\in\{-1,1\}$ for all $j,k\in\Z$. For an orthonormal
wavelet $\psi$ and a sequence $\eps=\{\eps_{j,k}\}_{j,k\in\Z}\in\cE$, 
consider the kernel
%%%
\begin{equation}\label{eq:wavelet-kernel}
K_\eps(x,y):=\sum_{j\in\Z}\sum_{k\in\Z}
\eps_{j,k}\psi_{j,k}(x)\overline{\psi_{j,k}(y)},
\quad x,y\in\R.
\end{equation}
%%%
If the wavelet $\psi$ has a radial decreasing $L^1$-majorant $W$, it follows from
\cite[Section~5.3, Lemma 3.12]{HW96} that
%%%
\begin{align*}
|K_\eps(x,y)|
&\le
\sum_{j\in\Z}2^j\sum_{k\in\Z}|\psi(2^jx-k)\psi(2^jy-k)|
\le 
\sum_{j\in\Z}2^j\sum_{k\in\Z}W(|2^jx-k|)W(|2^jy-k|) 
%\end{align*}
%\begin{align*}
\\
&\le 
C(W)\sum_{j\in\Z}2^j W(2^{j-1}|x-y|),
\end{align*}
%%%
where $C(W)$ depends only on $W$ and, by the proof of 
\cite[Section~5.6, Theorem~6.12]{HW96},
\[
\sum_{j\in\Z}2^jW(2^{j-1}|x-y|)\le\frac{4}{|x-y|}\|W\|_{L^1([0,\infty)}.
\]
Hence, $|K_\eps(x,y)|<\infty$ for all $x,y\in\R$ such that $x\ne y$.
%%%----------------------------------------------------------------------------
\begin{theorem}\label{th:wavelet-Hernandez-Weiss}
Suppose that $\psi$ is an orthonormal and differentiable wavelet such that
$\psi$ and its derivative $\psi'$ have a common radial decreasing 
$L^1$-majorant $W$ satisfying
%%%
\begin{equation}\label{eq:L1-majorant-condition}
\int_0^\infty sW(s)\,ds<\infty.
\end{equation}
%%%
Then the family $\{K_\eps\}_{\eps\in\cE}$ given by \eqref{eq:wavelet-kernel}
is a uniform in $\cE$ family of standard kernels with the constants 
$C_1, C_2, C_3$ in \eqref{eq:CZ-kernel-1}--\eqref{eq:CZ-kernel-3} 
depending only on $W$.
\end{theorem}
%%%----------------------------------------------------------------------------
The proof of this theorem is analogous to the proof of 
\cite[Section~5.6, Theorem~6.12]{HW96} and therefore it is omitted.

Let us consider the operator $T_{K_\eps}$ associated with $K_\eps$, which is
given for $f\in L^2(\R)$ by
%%%
\begin{equation}\label{eq:wavelet-CZ-operator}
(T_{K_\eps}f)(x)=\sum_{j\in\Z}\sum_{k\in\Z}
\eps_{j,k}\langle f,\psi_{j,k}\rangle \psi_{j,k}(x),
\quad x\in\R.
\end{equation}
%%%
For each $\eps=\{\eps_{j,k}\}_{j,k\in\Z}\in\cE$, the operator $T_{K_\eps}$ 
is an isometry on $L^2(\R)$.
%%%----------------------------------------------------------------------------
\begin{theorem}\label{th:wavelet-CZ-weak-type-1-1}
Suppose that $\psi$ is an orthonormal wavelet having a radial decreasing
$L^1$-majorant $W$ satisfying \eqref{eq:L1-majorant-condition}. There exists a 
constant $C_{1,1}(W)$ depending only on $W$ such that for every sequence 
$\eps=\{\eps_{j,k}\}_{j,k\in\Z}\in\cE$ and every function 
$f\in L^1(\R)\cap L^2(\R)$,
\[
\sup_{\lambda >0}
\left(\lambda\left|\left\{x\in\R\ :\ 
\left|\left(T_{K_\eps}f\right)(x)\right|>\lambda
\right\}\right|\right) 
\le C_{1,1}(W)\|f\|_{L^1(\R)},
\]
where the family of operators $\{T_{K_\eps}\}_{\eps\in\cE}$ is defined by
\eqref{eq:wavelet-kernel} and \eqref{eq:wavelet-CZ-operator} on $L^2(\R)$.
\end{theorem}
%%%----------------------------------------------------------------------------
Since the function $s\mapsto W(|s|)$ belongs to $L^\infty(\R)$ and,
by \eqref{eq:L1-majorant-condition},
\[
\int_0^\infty W(s)\ln(1+s)\,ds\le \int_0^\infty sW(s)\,ds<\infty,
\]
Theorem~\ref{th:wavelet-CZ-weak-type-1-1} follows from 
\cite[Chap.~7, Theorem~9]{KS99}.

Combining Theorems~\ref{th:Alvarez-Perez-uniform},
\ref{th:wavelet-Hernandez-Weiss}, \ref{th:wavelet-CZ-weak-type-1-1}
with Lemma~\ref{le:standard-kernel-implies-CD}, we arrive at the following.
%%%----------------------------------------------------------------------------
\begin{theorem}\label{th:wavelet-Alvarez-Perez}
Suppose that $\psi$ is an orthonormal and differentiable wavelet such that
$\psi$ and its derivative have a common radial decreasing $L^1$-majorant
$W$ satisfying \eqref{eq:L1-majorant-condition}. Then there exist constants 
$C_{1,1}(W), C_D(W)\in(0,\infty)$ depending only on $W$ such that for every 
$\eps=\{\eps_{j,k}\}_{j,k\in\Z}\in\cE$, every $s\in(0,1)$, every 
$f\in C_0^\infty(\R)$ and every $x_0\in\R$, one has
\[
\left(T_{K_\eps}f\right)_s^\#(x_0)\le C_s(W)(Mf)(x_0),
\]
where
\[
C_s(W):=2^{2/s-1}\big(2(1-s)^{-1/s}C_{1,1}(W)+C_D(W)\big),
\]
and the family of operators $\{T_{K_\eps}\}_{\eps\in\cE}$ is defined by
\eqref{eq:wavelet-kernel} and \eqref{eq:wavelet-CZ-operator}.
\end{theorem}
%%%----------------------------------------------------------------------------
\subsection{Boundedness of square functions associated with orthonormal 
wavelets} 
Let $\psi$ be an orthonormal wavelet. Consider the family of kernels
$\{K_\eps\}_{\eps\in\cE}$ defined by \eqref{eq:wavelet-kernel}
and the family of operators $\{T_{K_\eps}\}_{\eps\in\cE}$ associated with 
these kernels, which are given by \eqref{eq:wavelet-CZ-operator} for 
$f\in C_0^\infty(\R)$.
%%%----------------------------------------------------------------------------
\begin{theorem}\label{th:wavelet-CZ-uniform-boundedness}
Let $X(\R)$ be a separable Banach function space such that the Hardy-Littlewood
maximal operator $M$ is bounded on the space $X(\R)$ and on its associate
space $X'(\R)$. Suppose that $\psi$ is an orthonormal and differentiable 
wavelet such that $\psi$ and its derivative have a common radial decreasing 
$L^1$-majorant $W$ satisfying \eqref{eq:L1-majorant-condition}. Then for all
$\eps=\{\eps_{j,k}\}_{j,k\in\Z}\in\cE$, the operators $T_{K_\eps}$, defined 
initially on $C_0^\infty(\R)$, extend to bounded linear operators on $X(\R)$ 
and
\[
N:=\sup_{\eps\in\cE}\left\|T_{K_\eps}\right\|_{\cB(X(\R))}<\infty.
\]
\end{theorem}
%%%----------------------------------------------------------------------------
This theorem follows from Theorems~\ref{th:uniform-boundedness} and
\ref{th:wavelet-Alvarez-Perez}.

Let $X(\R)$ be a Banach function space and $X'(\R)$ be its associate space.
It follows from the H\"older inequality for Banach function spaces
(see \cite[Chap.~1, Theorem~2.4]{BS88}) that for $f\in X(\R)$ and 
$g\in X'(\R)$, the pairing $\langle f,g\rangle$ is correctly defined
by \eqref{eq:pairing}.

One of the main ingredients of the proof of the existence of
unconditional wavelet bases in Banach function spaces is the following 
theorem. Its proof is inspired by Meyer's approach (see 
\cite[Section~6.2]{M95} for the case of Lebesgue spaces and 
\cite[Theorem~4.2]{INS15} for weighted variable Lebesgue spaces).
%%%----------------------------------------------------------------------------
\begin{theorem}\label{th:operator-V}
Let $X(\R)$ be a separable Banach function space such that the Hardy-Littlewood
maximal operator $M$ is bounded on the space $X(\R)$ and on its associate
space $X'(\R)$. Suppose that $\psi$ is an orthonormal and differentiable 
wavelet such that $\psi$ and its derivative have a common radial decreasing 
$L^1$-majorant $W$ satisfying \eqref{eq:L1-majorant-condition} and such that 
$\psi_{j,k}\in X'(\R)$ for all $j,k\in\Z$. Then the sublinear operator $V$ 
defined by
%%%
\begin{equation}\label{eq:operator-V-1}
(Vf)(x):=
\left(
\sum_{j\in\Z}\sum_{k\in\Z}|\langle f,\psi_{j,k}\rangle \psi_{j,k}(x)|^2
\right)^{1/2},
\quad x\in\R,
\end{equation}
%%5
is bounded on the space $X(\R)$.
\end{theorem}
%%%----------------------------------------------------------------------------
\begin{proof}
Fix $f\in X(\R)$ and $g\in X'(\R)$ such that $\|g\|_{X'(\R)}\le 1$. Let
the set $\cE=\{-1,1\}^{\Z\times\Z}$ be equipped with the Bernoulli
probability measure $\mu$ obtained by taking the product of the measures 
on each factor $\{-1,1\}$, which give a mass $1/2$ to each of the points $-1$ 
and $1$. By Khintchine's inequality (see, e.g., \cite[Corollary~8.1]{Gut05}
and also \cite[Section~6.2, Lemma~2]{M95}), there exists a constant 
$L\in(0,\infty)$ such that
%%%
\begin{equation}\label{eq:operator-V-2}
(Vf)(x)\le 
L\int_\cE\left|\left(T_{K_\eps}f\right)(x)\right|d\mu(\eps),
\quad x\in\R,
\end{equation}
%%%
where the family of operators $\{T_{K_\eps}\}_{\eps\in\cE}$ is defined by
\eqref{eq:wavelet-kernel} and \eqref{eq:wavelet-CZ-operator}. Then, by 
inequality \eqref{eq:operator-V-2}, Fubini's theorem and H\"older's
inequality for Banach function spaces (see \cite[Chap.~1, Theorem~2.4]{BS88}),
we obtain
%%%
\begin{align}
\int_\R|(Vf)(x)g(x)|\,dx
&\le 
L\int_\R
\left(\int_\cE\left|\left(T_{K_\eps}f\right)(x)\right|d\mu(\eps)\right)
|g(x)|\,dx
\nonumber\\
&=
L\int_\cE\left(\int_\R
\left|\left(T_{K_\eps}f\right)(x)g(x)\right|dx
\right)d\mu(\eps)
\nonumber\\
&\le
L \int_\cE\left\|T_{K_\eps}f\right\|_{X(\R)}\|g\|_{X'(\R)}d\mu(\eps) 
\nonumber\\
&\le 
L \int_\cE\left\|T_{K_\eps}f\right\|_{X(\R)}d\mu(\eps).
\label{eq:operator-V-3}
\end{align}
%%%
By Theorem~\ref{th:wavelet-CZ-uniform-boundedness}, for all 
$\eps=\{\eps_{j,k}\}_{j,k\in\Z}\in\cE$, we have
%%%
\begin{equation}\label{eq:operator-V-4}
\left\|T_{K_\eps}f\right\|_{X(\R)}\le N\|f\|_{X(\R)}.
\end{equation}
%%%
Combining \eqref{eq:operator-V-3} and \eqref{eq:operator-V-4}, we see that
for all $f\in X(\R)$ and all $g\in X'(\R)$ satisfying $\|g\|_{X'(\R)}\le 1$,
\[
\int_\R|(Vf)(x)g(x)|\,dx 
\le 
LN\int_\cE\|f\|_{X(\R)}\,d\mu(\eps)=
LN\|f\|_{X(\R)}.
\]
It follows from the above inequality and the Lorentz-Luxemburg theorem
(see \cite[Chap.~1, Theorem~2.7]{BS88}) that for all $f\in X(\R)$,
%%%
\begin{align*}
\|Vf\|_{X(\R)}&=\|Vf\|_{X''(\R)}
\\
&=\sup\left\{
\int_\R|(Vf)(x)g(x)|\,dx\ :\ g\in X'(\R),\ \|g\|_{X'(\R)}\le 1
\right\}
\\
&\le LN\|f\|_{X(\R)},
\end{align*}
%%%
which completes the proof.
\end{proof}
%%%----------------------------------------------------------------------------
\subsection{Existence of a wavelet basis in a Banach function space}
By \cite[Section~2.3, Theorem~3.29]{HW96}, for every $r\in\{0,1,2\dots\}$,
there exists an orthonormal wavelet $\psi$ with compact support such that
$\psi$ has bounded derivatives up to order $r$.

Recall that a function $f$ in a Banach function space $X(\R)$ is said to have
absolutely continuous norm in $X(\R)$ if $\|f\chi_{E_n}\|_{X(\R)}\to 0$
for every sequence $\{E_n\}_{n=1}^\infty$ of measurable sets on $\R$
such that $\chi_{E_n}\to 0$ a.e. on $\R$ as $n\to\infty$. If all
functions $f\in X(\R)$ have this property, then the space $X(\R)$ itself
is said to have absolutely continuous norm 
(see \cite[Chap.~1, Definition~3.1]{BS88}).

Now we are in a position to prove the main result of this section.
%%%----------------------------------------------------------------------------
\begin{theorem}\label{th:wavelet-basis}
Let $X(\R)$ be a reflexive Banach function space such that the Hardy-Littlewood 
maximal operator $M$ is bounded on $X(\R)$ and on its associate space $X'(\R)$.
Suppose that $\psi$ is an orthonormal $C^1$-wavelet with compact support. Then 
the system 
\[
\{\psi_{j,k}:j,k\in\Z\} 
\]
is an unconditional basis in $X(\R)$ and the wavelet expansion
\[
f=\sum_{j\in\Z}\sum_{k\in\Z}\langle f,\psi_{j,k}\rangle \psi_{j,k}
\]
holds for every $f\in X(\R)$, where the convergence is unconditional in 
$X(\R)$.
\end{theorem}
%%%----------------------------------------------------------------------------
\begin{proof}
If a Banach function space $X(\R)$ is reflexive, then it follows from
\cite[Chap.~1, Corollary~4.4]{BS88} that the spaces $X(\R)$ and 
$X'(\R)$ have absolutely continuous norms. Hence, by 
\cite[Chap. 1, Corollary~5.6]{BS88}, the spaces $X(\R)$ and $X'(\R)$ are 
separable.
Then, in view of Lemma~\ref{le:density},
$L^2(\R)\cap X(\R)$ is dense in $X(\R)$ and $L^2(\R)\cap X'(\R)$ is dense
in $X'(\R)$. Since $\psi_{j,k}\in C_0(\R)$, we have $\psi_{j,k}\in X(\R)$ 
and $\psi_{j,k}\in X'(\R)$ for all $j,k\in\Z$. Moreover, there exist constants 
$C,\delta>0$ such that $W(s)=Ce^{-\delta s}$, $s\in[0,\infty)$, is a common 
$L^1$-majorant for $\psi$ and $\psi'$ that satisfies 
\eqref{eq:L1-majorant-condition}. By Theorem~\ref{th:operator-V}, the 
operator $V$ given by \eqref{eq:operator-V-1} is bounded on the spaces $X(\R)$ 
and $X'(\R)$. Then the desired result follows from \cite[Theorem~4.1]{INS15}.
\end{proof}
%%%----------------------------------------------------------------------------
\section{Noncompactness of multiplication and Fourier convolution operators}
\label{sec:noncompactness}
\subsection{Noncompactness of nontrivial multiplication operators}
The following theorem can be extracted from \cite[Theorem~2.4]{HKK06}.
%%%----------------------------------------------------------------------------
\begin{theorem}\label{th:noncompactness-multiplication}
Let $X(\R)$ be a separable Banach function space and $a\in L^\infty(\R)$. Then 
the  multiplication operator $aI$ is compact on the space $X(\R)$ if and only
if $a=0$ almost everywhere on $\R$.
\end{theorem}
%%%----------------------------------------------------------------------------
We give another proof of this result based on the following lemma, which is
of independent interest.

For a sequence of operators $\{A_n\}_{n\in\N}\subset\cB(X(\R))$, let
\[
\operatornamewithlimits{s-\lim}_{n\to\infty}A_n
\]
denote the strong limit of the sequence, if it exists.
For $\lambda,x\in\R$, consider the function 
\[
e_\lambda(x):=e^{i\lambda x}.
\]
%%%----------------------------------------------------------------------------
\begin{lemma}\label{le:LO-compact}
Let $X(\R)$ be a separable Banach function space and $K$ be a compact operator
on $X(\R)$. Then
\[
\operatornamewithlimits{s-\lim}_{n\to\infty}e_{h_n}Ke_{h_n}^{-1}I=0
\]
on the space $X(\R)$ for every sequence $\{h_n\}_{n\in\N}\subset\R$ such that
\[
\lim_{n\to\infty}h_n=\pm\infty.
\]
\end{lemma}
%%%----------------------------------------------------------------------------
\begin{proof}
The idea of the proof is borrowed from the proof of \cite[Lemma~10.1]{BKS02} 
and \cite[Lemma~3.8]{KJLH08}. Let $f\in X(\R)$ and $g\in X'(\R)$. By 
H\"older's inequality for Banach function spaces (see 
\cite[Chap.~1, Theorem~2.4]{BS88}),  $f\overline{g}\in L^1(\R)$. Hence, by 
the Riemann-Lebesgue lemma (see, e.g., \cite[Chap.~VI, Theorem~1.7]{Kat76}),
%%%
\begin{equation}\label{eq:LO-compact-1}
\lim_{n\to\infty}\int_\R e^{-ixh_n}f(x)\overline{g(x)}\,dx
=
\lim_{n\to\infty}(f\overline{g})\widehat{\hspace{2mm}}(-h_n)=0
\end{equation}
%%%
whenever $h_n\to\pm\infty$ as $n\to\infty$. Since the space $X(\R)$ is
separable, it follows from \cite[Chap.~1, Corollaries~4.3 and~5.6]{BS88}
that the associate space $X'(\R)$ is canonically isometrically isomorphic
to the Banach dual space $[X(\R)]^*$ of $X(\R)$. Hence equality 
\eqref{eq:LO-compact-1} implies that the sequence of multiplication operators 
$\{e_{h_n}^{-1}I\}_{n\in\N}$ converges weakly to the zero operator on the
space $X(\R)$ as $n\to\infty$. It is clear that 
$\|e_{h_n}I\|_{\cB(X(\R))}\le 1$ for all $n\in\N$. Since the sequence 
$\{e_{h_n}I\}_{n\in\N}$ is uniformly bounded, the operator $K$ is compact,
and the sequence $\{e_{h_n}^{-1}I\}_{n\in\N}$ converges weakly to the zero
operator as $n\to\infty$, we conclude that in view of 
\cite[Lemmas~1.4.4 and~1.4.6]{RSS11}, the sequence 
$\{e_{h_n}Ke_{h_n}^{-1}I\}_{n\in\N}$ converges strongly to the zero operator
on the space $X(\R)$ as $n\to\infty$.
\end{proof}
%%%----------------------------------------------------------------------------
\begin{proof}[Proof of Theorem~\ref{th:noncompactness-multiplication}]
It is clear that if $a=0$ a.e. on $\R$, then $aI$ is the zero operator, whence
it is compact. Assume that $aI$ is compact and consider a sequence
$\{h_n\}_{n\in\N}\subset\R$ such that $h_n\to+\infty$ as $n\to\infty$.
It is clear that $e_{h_n}(aI)e_{h_n}^{-1}I=aI$ for $n\in\N$. Then, by
Lemma~\ref{le:LO-compact}, 
\[
aI=\operatornamewithlimits{s-\lim}_{n\to\infty} e_{h_n}(aI)e_{h_n}^{-1}I=0,
\]
which implies that $a=0$ a.e. on $\R$.
\end{proof}
%%%----------------------------------------------------------------------------
\subsection{Noncompactness of nontrivial Fourier convolution operators}
The following result was recently obtained by the authors.
%%%----------------------------------------------------------------------------
\begin{theorem}[{\cite[Theorem~1.1]{FKK18}}]
Let $X(\R)$ be a separable Banach function space such that the
Hardy-Littlewood maximal operator $M$ is bounded on $X(\R)$ and on its
associate space $X'(\R)$. Suppose that $b\in\cM_{X(\R)}$. Then the Fourier
convolution operator $W^0(a)$ is compact on the space $X(\R)$ if and only
if $b=0$ almost everywhere on $\R$.
\end{theorem}
%%%----------------------------------------------------------------------------
In the next section we will show that, along with the fact that
each individual nontrivial multiplication operator $aI$ with $a\in L^\infty(\R)$
and each nontrivial Fourier convolution operator $W^0(b)$ with
$b\in\cM_{X(\R)}$ is never compact on the space $X(\R)$, the algebra generated 
by the operators $aI$ and $W^0(b)$ contains all compact operators,
similarly to Banach algebras of Toeplitz operators with continuous symbols on 
Hardy spaces.
%%%----------------------------------------------------------------------------
\section{Algebra of convolution type operators with continuous data}
\label{sec:algebra}
%%%----------------------------------------------------------------------------
\subsection{Fourier convolution operators with symbols in the algebra 
\boldmath{$V(\R)$}}
Suppose that $a:\R\to\C$ is a function of finite total variation $V(a)$ given 
by
\[
V(a):=\sup \sum_{k=1}^n |a(x_k)-a(x_{k-1})|,
\]
where the supremum is taken over all partitions of $\R$ of the form
\[
-\infty<x_0<x_1<\dots<x_n<+\infty
\]
with $n\in\N$. The set $V(\R)$ of all functions of finite total variation
on $\R$ with the norm
\[
\|a\|_V:=\|a\|_{L^\infty(\R)}+V(a)
\]
is a unital non-separable Banach algebra.
%%%----------------------------------------------------------------------------
\begin{theorem}\label{th:Stechkin}
Let $X(\R)$ be a separable Banach function space such that the
Hardy-Littlewood maximal operator $M$ is bounded on $X(\R)$ and on its
associate space $X'(\R)$. If a function $a:\R\to\C$ has a finite total 
variation $V(a)$, then the convolution operator $W^0(a)$ is bounded on 
the space $X(\R)$ and
%%%
\begin{equation}\label{eq:Stechkin}
\|W^0(a)\|_{\cB(X(\R))}
\le
c_{X}\|a\|_V
\end{equation}
%%%
where $c_{X}$ is a positive constant depending only on $X(\R)$.
\end{theorem}
%%%----------------------------------------------------------------------------
This result follows from \cite[Theorem~4.3]{K15a}.

For Lebesgue spaces $L^p(\R)$, $1<p<\infty$, inequality~\eqref{eq:Stechkin} is
usually called Stechkin's inequality, and the constant $c_{L^p}$ is
calculated explicitly:
%%%
\begin{equation}\label{eq:constant-in-Stechkin}
c_{L^p}=\|S\|_{\cB(L^p(\R))}=\left\{\begin{array}{ccc}
\tan\left(\frac{\pi}{2p}\right) &\mbox{if}& 1<p\le 2,
\\[3mm]
\cot\left(\frac{\pi}{2p}\right) &\mbox{if}& 2\le p<\infty,
\end{array}\right.
\end{equation}
%%%
where $S$ is the Cauchy singular integral operator given by
%%%
\begin{equation}\label{eq:Cauchy-singular-integral-operator}
(Sf)(x):=\frac{1}{\pi i}\lim_{\eps\to 0}\int_{\R\setminus(x-\eps,x+\eps)}
\frac{f(t)}{t-x}\,dt.
\end{equation}
%%%
We refer to \cite[Theorem~2.11]{D79} for the proof of \eqref{eq:Stechkin}
in the case of Lebesgue spaces $L^p(\R)$ with $c_{L^p}=\|S\|_{\cB(L^p(\R))}$
and to \cite[Chap. 13, Theorem 1.3]{GK92} for the calculation of the norm
of $S$ given in the second equality in \eqref{eq:constant-in-Stechkin}.
%%%----------------------------------------------------------------------------
\subsection{One-dimensional operator with continuous compactly supported data}
A proof of the next lemma can be extracted from the proof of 
\cite[Lemma~6.1]{KILH13}.
%%%----------------------------------------------------------------------------
\begin{lemma}\label{le:one-dimensional-operator}
Suppose $X(\R)$ is a separable Banach function space. Let $a,b\in C_0(\R)$ and 
an one-dimensional operator $T_1$ be defined on the space $X(\R)$ by
%%%
\begin{equation}\label{eq:one-dimensional-operator}
(T_1f)(x)=a(x)\int_\R b(y)f(y)\,dy.
\end{equation}
%%%
Then there exists a function $c\in C(\dR)\cap V(\R)$ such that $T_1=aW^0(c)bI$.
\end{lemma}
%%%----------------------------------------------------------------------------
\subsection{Proof of Theorem~\ref{th:main}}
It follows from Theorem~\ref{th:wavelet-basis} that the space $X(\R)$ has a 
Schauder basis. It is well known that every compact operator on a Banach space
with a Schauder basis can be approximated in the operator norm by linear
operators of finite rank (see, e.g., \cite[Chap.~I, Corollary~17.7]{S70}). 
It follows from \cite[Chap.~1, Corollaries 4.3 and 4.4]{BS88} that the Banach 
space dual $[X(\R)]^*$ of the space $X(\R)$ is canonically isometrically 
isomorphic to the associate space $X'(\R)$. Hence a finite rank operator on 
$X(\R)$ is of the form
%%%
\begin{equation}\label{eq:ideal-into-algebra-1}
(T_mf)(x)=\sum_{j=1}^m a_j(x)\int_\R b_j(y)f(y)\,dy,\quad x\in\R,
\end{equation}
%%%
where $a_j\in X(\R)$ and $b_j\in X'(\R)$ for $j\in\{1,\dots,m\}$ and some
$m\in\N$. Since the set $C_0(\R)$ is dense in $X(\R)$ and in $X'(\R)$
in view of Lemma~\ref{le:density}, for every $\eps\in(0,1)$
and every $j\in\{1,\dots,m\}$, there exist $a_{j,\eps},b_{j,\eps}\in C_0(\R)$
such that
%%%
\begin{equation}\label{eq:ideal-into-algebra-2}
\big|\|a_j\|_{X(\R)}-\|a_{j,\eps}\|_{X(\R)}\big|<1
\end{equation}
%%%
and
%%%
\begin{equation}\label{eq:ideal-into-algebra-3}
\|a_j-a_{j,\eps}\|_{X(\R)}<\frac{\eps}{2m(\|b_j\|_{X'(\R)}+1)},
\quad
\|b_j-b_{j,\eps}\|_{X'(\R)}<\frac{\eps}{2m(\|a_j\|_{X(\R)}+1)}.
\end{equation}
%%%
Let $T_{m,\eps}$ denote the operator defined by \eqref{eq:ideal-into-algebra-1}
with $a_{j,\eps}$ and $b_{j,\eps}$ in place of $a_j$ and $b_j$, respectively.
It follows from H\"older's inequality for Banach function spaces 
(see \cite[Chap.~1, Theorem~2.4]{BS88}) and 
inequalities \eqref{eq:ideal-into-algebra-2}--\eqref{eq:ideal-into-algebra-3}
that for $f\in X(\R)$,
%%%
\begin{align*}
&
\|T_mf-T_{m,\eps}f\|_{X(\R)}
\\
&\quad
\le 
\left\|
\sum_{j=1}^m (a_j-a_{j,\eps})\int_\R b_j(y)f(y)\,dy
\right\|_{X(\R)}
+
\left\|
\sum_{j=1}^m a_{j,\eps}\int_\R(b_j(y)-b_{j,\eps}(y))f(y)\,dy
\right\|_{X(\R)}
\\
&\quad
\le
\sum_{j=1}^m\|a_j-a_{j,\eps}\|_{X(\R)}\|b_j\|_{X'(\R)}\|f\|_{X(\R)}
+
\sum_{j=1}^m\|a_{j,\eps}\|_{X(\R)}\|b_j-b_{j,\eps}\|_{X'(\R)}\|f\|_{X(\R)}
\\
&\quad
<
\sum_{j=1}^m\frac{\eps \|f\|_{X(\R)}}{2m(\|b_j\|_{X'(\R)}+1)}\|b_j\|_{X'(\R)}
+
\sum_{j=1}^m(\|a_j\|_{X(\R)}+1)\frac{\eps \|f\|_{X(\R)}}{2m(\|a_j\|_{X(\R)}+1)}
\\
&\quad
<\eps\|f\|_{X(\R)},
\end{align*}
%%%
whence $\|T_m-T_{m,\eps}\|\le\eps$. Therefore, each compact operator on the 
space $X(\R)$ can be approximated in the operator norm by a finite sum
of rank one operators $T_1$ of the form \eqref{eq:one-dimensional-operator}
with $a,b\in C_0(\R)$. By Lemma~\ref{le:one-dimensional-operator}, each
such operator can be written in the form $T_1=aW^0(c)bI$ with 
$c\in C(\dR)\cap V(\R)$. It follows from Theorem~\ref{th:Stechkin} that
$c\in C_X(\dR)$. Hence $T_1\in\cA_{X(\R)}$,
which completes the proof.
\rule{1.5mm}{1.5mm}
%%%----------------------------------------------------------------------------
\subsection*{Acknowledgments}
This work was partially supported by the Funda\c{c}\~ao para a Ci\^encia e a
Tecnologia (Portu\-guese Foundation for Science and Technology)
through the project
UID/MAT/00297/2019 (Centro de Matem\'atica e Aplica\c{c}\~oes).
We are grateful to the referees for the useful comments and suggestions.
%%%----------------------------------------------------------------------------

\end{document}